\documentclass[12 pt, leqno]{article}
\usepackage{amssymb,amsmath,amsfonts,amsthm}
\usepackage{accents}
\usepackage[none]{hyphenat}

\DeclareMathAlphabet{\chan}{T1}{pzc}{mb}{it}

\newcommand{\op}{\operatorname}

\newcommand{\mc}{\mathcal}
\newcommand{\boldsigma}{\mathbf{\mathop{\pmb{\sum}}}}

\newtheorem{thm}{Theorem}

\newsavebox{\Theoremtype}
\newsavebox{\Theoremlabel}

\newtheoremstyle{ref}
{\topsep}	
{\topsep}	
{\it}
{}
{}
{}
{ }
{\thmname{{\bfseries#1}}\thmnumber{ \bfseries#2\thmnote{\rm #3}\bfseries .}}

\theoremstyle{ref}
\newtheorem{thrm}[thm]{Theorem}
\newtheorem{lem}[thm]{Lemma}
\newtheorem{prop}[thm]{Proposition}
\newtheorem{cor}[thm]{Corollary}

\newtheoremstyle{nnref}
{\topsep}	
{\topsep}	
{\it}
{}
{}
{}
{ }
{\thmname{\textbf{#1}\thmnote{\textrm{ #3}}\textbf{.}}}

\theoremstyle{nnref}
\newtheorem{defn}{Definition}

\begin{document}
\sloppy
\title{$\mathrm{PFA}(S)[S]$ and the Arhangel'ski\u\i-Tall problem}
\author{Franklin D. Tall\makebox[0cm][l]{$^1$}}

\footnotetext[1]{Research supported by grant A-7354 of the Natural Sciences and Engineering Research Council of Canada.\vspace*{1pt}}
\date{\today}
\maketitle

\renewcommand{\thefootnote}{}
\footnote
{\parbox[1.8em]{\linewidth}{$(2010)$ Mathematics Subject Classification. Primary 54A35, 54D15, 54D45; Secondary 03E35, 03E55, 03E57, 03E75.}\vspace*{3pt}}
\renewcommand{\thefootnote}{}
\footnote
{\parbox[1.8em]{\linewidth}{Key words and phrases: locally compact, normal, perfectly normal, metacompact, metalindel\"of, collectionwise Hausdorff, $\op{PFA}(S)[S]$, forcing with a coherent Souslin tree, Axiom $R$, supercompact cardinal.}}

\begin{abstract}
We discuss the Arhangel'ski\u\i-Tall problem and related questions
in models obtained by forcing with a coherent Souslin tree.
\end{abstract}

\section{Introduction}

Around 1965, A.V.~Arhangel'ski\u\i\ proved:
\begin{prop}
Every locally compact, perfectly normal, metacompact space is paracompact.
\end{prop}

In response to my question to him in Prague in 1971 as to whether this
was true, he responded that he had proved it, but his mentor,
P.~S.~Alexandrov, had not thought it worth publishing!  He
subsequently published it in \cite{Arhangelskii1972}.  Neither of us
could answer the question of what happened if the ``closed sets are
$G_\delta$'s'' requirement was dropped.  I raised this in
\cite{Tall1974} and it became known as the ``Arhangel'ski\u\i-Tall''
problem.  A partial solution was achieved in \cite{Watson1982}, where
S.~Watson proved:

\begin{prop}\label{prop2}
  $V=L$ implies every locally compact normal metalindel\"of space is
  paracompact.
\end{prop}

Then G.~Gruenhage and P.~Koszmider \cite{Gruenhage1996b} proved:
\begin{prop}
  If $ZFC$ is consistent, it is consistent with the existence of a
  locally compact, normal, metacompact space which is not paracompact.
\end{prop}

In 2003, using results announced by S.~Todorcevic (which have now been proved in \cite{LarsonErice} plus \cite{Fischer}), P.~Larson and I
\cite{Larson} proved:
\begin{thm}\label{thm4}
  If the existence of a supercompact cardinal is consistent with
  $ZFC$, so is the assertion that every locally compact, perfectly
  normal space is paracompact.
\end{thm}

The question remained as to whether one could obtain the paracompactness of
locally compact normal metacompact spaces as well as the conclusion of Theorem
\ref{thm4} in the same model, i.e.~could we change Arhangel'ski\u\i's
``and'' to an ``or''?  That is what we shall do here, subject to the
same large cardinal assumption as in Theorem \ref{thm4}.  We
conjecture that that assumption can be eliminated.

\begin{thm}\label{thm4a}
If the existence of a supercompact cardinal is consistent with $ZFC$, so
is the assertion that every locally compact normal space that either
is metalindel\"{o}f or has all closed sets $G_\delta$'s is paracompact.
\end{thm}

\section{$\mathrm{PFA}(S)[S]$}

The model that we
use first is the same one as for Theorem \ref{thm4}.  I use the
convention that ``$\op{PFA}(S)[S]$ implies $\Phi$'' stands for the
assertion that in any model constructed by starting with a coherent
Souslin tree $S$, forcing to obtain $\op{PFA}(S)$, i.e.~$\op{PFA}$ restricted to
proper posets preserving $S$, and then forcing with $S$, $\Phi$ holds.
We analogously use ``$\op{MA}_{\omega_1}(S)[S]$''.  For a discussion of
such models and the definition of \textit{coherent}, see
\cite{Larson2001}.  The model of Theorem \ref{thm4} is a model
constructed in that fashion, but over a particular ground model.  We
shall later discuss modifications of that model.

That the model of Theorem \ref{thm4} suffices to prove Theorem \ref{thm4a} follows immediately from:

\begin{thrm}[~\cite{Tall}]
  PFA$(S)[S]$ implies locally compact normal spaces are
  $\aleph_1$-collectionwise Hausdorff.
\end{thrm}

\begin{lem}[~\cite{Gruenhage1996}]
  Locally compact normal $\aleph_1$-collectionwise Hausdorff
  metalindel\"of spaces are paracompact.
\end{lem}
\qed

Gruenhage and Kozmider proved in \cite{Gruenhage1996} that:
\begin{prop}\label{prop8}
  $\op{MA}_{\omega_1}$ implies every locally compact, normal,
  metalindel\"of space is paracompact.
\end{prop}

Their only use of $\op{MA}_{\omega_1}$ was to prove the following proposition:

\begin{prop}\label{prop9}
  Assume $\op{MA}_{\omega_1}$.  Let $\{B_\alpha:\alpha<\omega_1\}$ be a
  collection of sets such that whenever $\{F_\alpha: \alpha <
  \omega_1\}$ is a disjoint collection of finite subsets of
  $\omega_1$, $\{\bigcup B_{\beta}: \beta \in F_\alpha, \alpha <
  \omega_1\}$ is not centered.  Let $\{Y_\alpha:\alpha<\omega_1\}$ be
  a collection of countable sets such that $|Y_\alpha -
  \bigcup\{B_\beta:\beta\in F\}| = \aleph_0$, for every finite $F
  \subseteq \omega_1-\{\alpha\}$.  Then $\omega_1 =
  \bigcup_{n<\omega}A_n$, where for each $n \in \omega$ and $\alpha
  \in \omega_1$, $|Y_\alpha - \bigcup\{B_{\beta}:\beta\in
  A_n-\{\alpha\}\}| = \aleph_0$.
\end{prop}

In fact, analyzing their use of Proposition \ref{prop9} in their
proof, we observe that they only needed that each stationary
$S\subseteq\omega_1$ included a stationary $T$ such that for every
$\alpha\in T$, $|Y_\alpha\,-\,\bigcup\{B_{\beta}:\beta\in
T-\{\alpha\}\}| = \aleph_0$.  This follows from there being a closed
unbounded $C$ such that for $\alpha\in C$, $|Y_\alpha -
\bigcup\{B_{\beta}:\beta\in C-\{\alpha\}\}|=\aleph_0$. We conjecture
this follows from PFA$(S)[S]$.

Watson \cite{Watson1986} constructed a locally compact, perfectly normal,
metalindel\"of, non-paracompact space from
$\op{MA}_{\omega_1}(\sigma$-centered$)$ plus the existence of a Souslin
tree. It is certainly consistent that there are no locally compact,
perfectly normal, metalindel\"of spaces that are not paracompact. It follows from Proposition \ref{prop8}, and in
fact I showed that $\op{MA}_{\omega_1}$ implied there weren't any a long
time ago in \cite{Tall1974}.

\section{Weakening the model of Theorem \ref{thm4}}

One wonders whether all of the requirements of the model of Theorem
\ref{thm4} are necessary.  Whether large cardinals are necessary has not yet been investigated. I conjecture that they are not, except possibly for an inaccessible.  Avoiding that issue, two others remain:
\begin{enumerate}
\item Is the preliminary forcing used in \cite{Larson} before forcing
  $\op{PFA}(S)[S]$ necessary?
\item Do we just need $\op{PFA}(S)[S]$, or do we need a model of
  $\op{PFA}(S)[S]$ constructed by the usual iteration, i.e. following the
  usual proof of the consistency of $\op{PFA}$, but using only those
  partial orders that preserve $S$ \cite{Miyamoto1993}?
\end{enumerate}

We can answer the first question negatively; our particular answer
however requires the second alternative for the second question.  The
preliminary forcing in \cite{Larson} --- adding $\lambda^+$ Cohen
subsets of $\lambda$ for every regular $\lambda \ge$ a supercompact
$\kappa$ --- was done so as to assure we could get full collectionwise
Hausdorffness from the $\aleph_1$-collectionwise Hausdorffness
provided by the Souslin tree forcing. An old consistency result of
Shelah \cite{Shelah1977} recast as a proof from a reflection axiom
\cite{Fleissner1986}, \cite{Dow1992}, \cite{Fuchino}, tells us that
under such an axiom, locally separable, first countable,
$\aleph_1$-collectionwise Hausdorff spaces are collectionwise
Hausdorff.  However, such reflection axioms do not follow from
$\op{PFA}(S)[S]$, but require a stronger principle holding in the usual
iteration model for $\op{PFA}(S)[S]$.  Now for the details.

First of all, the relevance of ``local separability'' is:
\begin{lem}
  If every first countable, hereditarily Lindel\"of, regular space is
  hereditarily separable, then locally compact perfectly normal spaces
  are locally separable.
\end{lem}

\begin{proof}
To see this, note that locally compact,
perfectly normal spaces are first countable. Next, note:

\begin{lem}[~\cite{Larson2002}]
  $\op{MA}_{\omega_1}(S)[S]$ implies every first countable, hereditarily
  Lindel\"of, regular space is hereditarily separable.
\end{lem}
\end{proof}

$\op{MA}_{\omega_1}(S)[S]$ of course follows from $\op{PFA}(S)[S]$.  There are
two reflection axioms in the literature we want to focus on, but we
will not actually need their complicated definitions.  ``Axiom $R$''
was introduced by Fleissner \cite{Fleissner1986}, who proved it
implied \emph{locally separable, first countable,
  $\aleph_1$-collectionwise Hausdorff spaces are collectionwise
  Hausdorff}. In \cite{Fuchinoa}, the authors interpo- lated a new
axiom, $FRP$, obtaining:

\begin{lem}[~\cite{Fuchino}, \cite{Fuchinoa}]
  Axiom $R$ implies $FRP$, which implies every locally separable,
  first countable, $\aleph_1$-collectionwise Hausdorff space is
  collectionwise Hausdorff.
\end{lem}
\begin{defn}[~\cite{Beaudoin1987}]
  $\mathbf{MA_{\boldsymbol{\omega}_1}(}$\textbf{countably closed,}
    $\boldsymbol{\kappa}\mathbf{)}$ is the assertion that if $P$ is a countably
  closed partial order, $\mc{D}$ is a family of at most $\aleph_1$
  dense subsets of $P$, and $\{S_\alpha:\alpha<\kappa\}$ is a family
  of cardinality $\kappa$ of $P$-terms, each forced by every condition
  in $P$ to denote a stationary subset of $\omega_1$, then there is a
  $\mc{D}$-generic filter $G$ on $P$ so that for every
  $\alpha<\kappa$, $S_\alpha(G)$ is stationary, where:
\[ S_\alpha(G) = \{\beta<\omega_1:(\exists p \in G)p\Vdash\check{\beta}\in\dot{S}_\alpha\}. \]
\end{defn}

$\mathbf{MA_{\boldsymbol{\omega}_1}(}$\textbf{proper,}
$\boldsymbol{\kappa}\mathbf{)}$ is defined analogously.  Baumgartner
\cite{Baumgartner1984} denotes $\op{MA}_{\omega_1}$(proper, $\aleph_1$) by
$\op{PFA}^+$; some authors use $\op{PFA}^{++}$ for ``$\op{MA}_{\omega_1}$(proper,
$\aleph_1$)'' and ``$\op{PFA}^+$'' for ``$\op{MA}_{\omega_1}$(proper, $1$)''.
We shall use Baumgartner's notation.

It is known that:
\begin{lem}[~\cite{Baumgartner1984}]
  $\op{PFA}^+$ holds in the usual iteration model for $\op{PFA}$.
\end{lem}
\begin{lem}[~\cite{Beaudoin1987}]
  $\op{MA}_{\omega_1}$(countably closed, $1$) implies Axiom $R$.
\end{lem}

Since countably closed partial orders preserve Souslin trees, we see
that $\op{MA}_{\omega_1}$(countably closed, 1)$(S)$ also implies Axiom $R$
and so $\op{PFA}^+$ does as well.  It is not known whether
$\op{MA}_{\omega_1}$(proper, 1)$(S)[S]$ implies Axiom $R$,
but $\op{PFA}^+(S)[S]$ implies Axiom $R$ \cite{LarsonOTHP}.

For $FRP$, there is a less specialized result:
\begin{lem}[~\cite{Fuchino}]
$FRP$ is preserved by countable chain condition forcing.
\end{lem}
\begin{cor}
$\op{MA}_{\omega_1}$(countably closed, 1)$(S)[S]$ implies $FRP$.
\end{cor}

We cannot, however, drop the one remaining stationary set:
\begin{thm}
  $\op{PFA}(S)[S]$ does not imply first countable, locally separable,
  $\aleph_1$-collectionwise Hausdorff spaces are collectionwise
  Hausdorff.
\end{thm}
\begin{proof}
  Beaudoin notes that he and M. Magidor have independently shown that
  $\op{PFA}$ is consistent with the existence of a non-reflecting
  stationary $E\subseteq\{\alpha<\omega_2:cf(\alpha)=\omega\}$.  Such
  a set is well-known to yield a ladder system space which is first
  countable, locally separable, $\aleph_1$-collectionwise Hausdorff
  but not collectionwise Hausdorff \cite{Fleissner1978}.  It thus only
  remains to show that such a space is preserved by the adjunction of
  a Souslin tree.  The first two properties are ``basis properties''
  and are clearly preserved.  For the space to become collectionwise
  Hausdorff, the stationarity of $E$ would have to be destroyed, which
  countable chain condition forcing can't do.  It remains to show that $\aleph_1$-collectionwise Hausdorffness is preserved.  The reason is that, by a standard argument, every subset $Y$ of size $\aleph_1$ of a ground model set $X$ in a c.c.c. extension is included in a ground model subset $Z$ of $X$ of size $\aleph_1$.  The ground model separation of $Z$ then restricts to a separation of $Y$.
\end{proof}

Note, however, that we have not proved that $\op{PFA}(S)[S]$ does not imply
locally compact, perfectly normal spaces are collectionwise
Hausdorff. We conjecture that this can be accomplished by proving that
a finite condition variant of the partial order Shelah
uses in \cite{Shelah1977} to force the ladder system space mentioned
above to be normal is proper and preserves $S$. Ladder system
spaces are locally compact Moore spaces and hence
have closed sets $G_\delta$, so that would suffice.

In addition to needing that first countable, hereditarily Lindel\"of,
regular spaces are hereditarily separable, and that locally compact,
perfectly normal spaces are collectionwise Hausdorff, the proof in
\cite{Larson} of the consistency of locally compact, perfectly normal
spaces being paracompact needed:

\begin{lem}[~\cite{Todorcevic},\cite{Todorcevic2008}]
$\op{PFA}(S)[S]$ implies $\boldsigma$.
\end{lem}
\begin{defn}
  \textbf{Balogh's $\boldsigma$} is the assertion that if $Y$ is a
  subset of size $\aleph_1$ of a compact, countably tight space $X$,
  and there is a family $\mc{V}$ of $\aleph_1$ open sets covering $Y$
  such that for every $V \in \mc{V}$ there is an open $U_V\subseteq X$
  such that $\overline{V}\subseteq U_V$ and $U_V\cap Y$ is countable,
  then $Y$ is $\sigma$-discrete.
\end{defn}

The status of Todorcevic's proof is as follows.  He announced the
result in a seminar in Toronto in 2002 and in a lecture in Prague in
2006.  He sketched the proof of the hardest part --- that $\op{PFA}(S)[S]$
implies that compact, countably tight spaces are sequential --- in his
lectures in Erice in 2008 \cite{Todorcevic2008}. He sketched a proof
of a weaker version of Balogh's $\boldsigma$ restricted to compact
sequential spaces in notes in 2002 \cite{Todorcevic}.   A proof that avoids the necessity for proving compact
countably tight spaces are sequential now exists in the union of \cite{LarsonErice} plus \cite{Fischer}.


\section{$\op{MA}_{\omega_1}(S)[S]$ does not imply there are no first
countable $S$-spaces}

The following material deals with a question analogous to what we have
considered so far: does a result proved to hold in a particular model
of $\op{MA}_{\omega_1}(S)[S]$ actually follow from $\op{MA}_{\omega_1}(S)[S]$?

A key unresolved question is \textit{whether $\mathrm{PFA}(S)[S]$
  implies there are no first countable $S$-spaces}. I had incorrectly
claimed this at a couple of conferences in 2006. If this is true, it
would follow that PFA$(S)[S]$ implies there are no first countable,
hereditarily normal, separable Dowker spaces. This is because of the
following proposition from~\cite{Larson}:

\begin{prop}
  {\rm MA}$_{\omega_1}\!(S)[S]$ implies first countable
  hereditarily normal spaces satisfying the countable chain condition
  are hereditarily separable.
\end{prop}

We shall now show that MA$_{\omega_1}\!(S)[S]$ is not sufficient to
prove there are no first countable
$S$-spaces. 
\begin{thm}
\label{ccc_thm}
Assume $2^{\aleph_1} = \aleph_2$. There is a c.c.c.\ poset $Q$ of size
$\aleph_2$ such that after forcing with $Q$ and then any c.c.c.\ poset $P$, there
is a first countable perfectly normal hereditarily separable space
which is not Lindel\"{o}f.
\end{thm}

This does it, since one can start e.g., with $L$ and in the $Q$
extension let $P = P_1\ast (\dot{P}_2\times S)\dot{}$, where $P_1$ is
the forcing for adding a Cohen real, which forces a coherent Souslin
tree $S$ \cite{Todorcevic1987}, and $\dot{P}_2$ forces
MA$_{\omega_1}\!(S)$. Since
$\dot{P}_2$ preserves $S$, $\dot{P}_2 \times S$ and hence $P$ is
c.c.c.. In order to force MA$_{\omega_1}\!(S)$ we need $2^{\aleph_1}
\leq\aleph_2$, but $Q\ast P_1$ preserves this. \qed

To see that Theorem~\ref{ccc_thm} holds we need just to assemble
results of others.
\begin{lem}[~\cite{Larson1999}]
{\rm MA}$_{\omega_1}\!(S)[S]$ implies $\mathfrak{b}>\aleph_1$.
\end{lem}
\begin{lem}[~\cite{douwen1984}]
\label{doscounter}
$\mathfrak{b}>\aleph_1$ implies that in a first countable regular space of size $\aleph_1$, two disjoint closed sets, one of which is countable, have disjoint open sets around them.
\end{lem}
\begin{lem}[~\cite{Soukup2001}]
  $2^{\aleph_1} = \aleph_2$ implies there is a c.c.c.\ poset $Q$ of
  size $\aleph_2$ such that after forcing with $Q$ and then any c.c.c.\ poset $P$,
  there is a first countable $0$-dimensional space of size $\aleph_1$
  in which every open set is countable or cocountable.
\end{lem}

It just remains to show Soukup's space has the desired properties in
our model. He notes it is hereditarily separable but not Lindel\"{o}f;
by Lemma~\ref{doscounter}, it is hereditarily normal. Without loss of
generality, by passing to a subspace if necessary, we may assume the
space is locally countable. But as Roitman~\cite{Roitman1984} notes on
p.\ 314, \textit{a countable subset of a locally countable space is a
  $G_\delta$}. Since cocountable sets are also $G_\delta$'s we see
that the space is perfectly normal.

\section{A problem of Nyikos}

Next, we deal with a tangentially related problem.
In~\cite{Nyikos1992}, Nyikos raises the question of whether
there is a separable, hereditarily normal, locally compact space of
cardinality $\aleph_1$. He observes that a model in which there are no
$Q$-sets and no locally compact first countable $S$-spaces would have
no such space. In~\cite{Eisworth2003}, such a model is
produced. PFA$(S)[S]$ also implies these two assertions, so it also
implies that there is no such space.  To see this, note that a $Q$-set enables the construction of a locally compact normal space which is not $\aleph_1$-collectionwise Hausdorff, while $\boldsigma$ implies there are no (locally) compact $S$-spaces. \qed

\section{Some Problems}
The referee has asked whether PMEA (the Product Measure Extension Axiom) implies
locally compact perfectly normal spaces are
paracompact, noting that PMEA implies locally compact normal metalindel\"{o}f spaces are
paracompact \cite{D}. I do not know the answer to this. However, the reason PMEA implies
locally compact normal metalindel\"{o}f spaces are paracompact is simply that it implies
normal spaces of character $< 2^{\aleph_0}$ are collectionwise normal, whence one gets
locally compact normal spaces are $\aleph_1$-collectionwise Hausdorff by the usual
Watson reduction \cite{Watson1982}. As noted earlier, that is enough to make locally compact
normal metalindel\"{o}f spaces paracompact. What one would need in addition to
$\aleph_1$-collectionwise Hausdorffness in order to make locally compact perfectly normal
spaces paracompact is the non-existence of compact $L$-spaces plus $\boldsigma$. It
is not known if the former holds under PMEA, although both it and PMEA will hold if
one adds strongly compact many random reals over a model of $\op{MA}_{\omega_1}$. (The latter
is a well-known result of Kunen (see \cite{F}), while the former is in \cite{T}.) The question
of whether $\boldsigma$ holds in this model is a stronger version of the unsolved
problem of whether there are compact $S$-spaces in the model obtained by adjoining
$\aleph_2$ random reals to a model of $\op{MA}_{\omega_1}$. For several years, this was the
preferred approach toward solving Kat\v etov's problem, before $\op{MA}_{\omega_1}(S)[S]$ turned
out to be the way to go \cite{Larson2002}.

The referee also asked whether the Abraham-Todorcevic example of a first countable
$S$-space indestructible under countable chain condition forcing \cite{AT} exists under
$\op{MA}_{\omega_1}(S)[S]$. Indeed, they start with a model of GCH and do a countable chain
condition iteration to construct their example. One can then force with a countable
chain condition poset to get $\op{MA}_{\omega_1}(S)[S]$, so this gives another proof that
$\op{MA}_{\omega_1}(S)[S]$ does not imply there are no first countable $S$-spaces. I conjecture
however that $\op{PFA}(S)[S]$ implies there are no $S$-spaces.

\nocite{*}
\bibliographystyle{acm}
\bibliography{meta}

\noindent
{\rm Franklin D. Tall\\
Department of Mathematics\\
University of Toronto\\
Toronto, Ontario\\
M5S 2E4\\
CANADA\\}
\noindent
{\it e-mail address:} {\rm f.tall@utoronto.ca}

\end{document}